\documentclass[11pt]{amsart}  \usepackage{verbatim}
\usepackage{rotating} 
\usepackage[bookmarks,colorlinks,breaklinks]{hyperref}
\usepackage{amssymb} \usepackage{color} \usepackage{amsmath, amscd}

\numberwithin{equation}{section}

\newtheorem{lemma}[equation]{Lemma}
\newtheorem{thm}[equation]{Theorem}
\newtheorem{conjecture}[equation]{Conjecture}

\newtheorem{prop}[equation]{Proposition}

\theoremstyle{remark}

\newtheorem{remark}[equation]{Remark}

\theoremstyle{remark}

\newcommand{\C}{{\mathbb C}}

\newcommand{\Q}{{\mathbb Q}}

\newcommand{\Z}{{\mathbb Z}}

\newcommand{\Obar}{\overline{O}}

\newcommand{\Qbar}{\overline{\Q}}

\DeclareMathOperator{\N}{\mathbb{N}}

\DeclareMathOperator{\ord}{ord}

\DeclareMathOperator{\Div}{Div}

\DeclareMathOperator{\GCD}{GCD}

\newcommand{\bC}{{\mathbb C}}

\newcommand{\bG}{{\mathbb G}}

\newcommand{\lra}{\longrightarrow}

\newcommand{\cE}{\mathcal{E}}

\newcommand{\hhat}{{\widehat h}}

  {\end{list}}

\begin{document}

\title{A variant of a theorem by Ailon-Rudnick for elliptic curves}

\author{D.~Ghioca}
\address{
Dragos Ghioca\\
Department of Mathematics\\
University of British Columbia\\
Vancouver, BC V6T 1Z2\\
Canada
}
\email{dghioca@math.ubc.ca}

\author{L.-C.~Hsia}
\address{
Liang-Chung Hsia\\
Department of Mathematics\\
National Taiwan Normal University\\
Taipei, Taiwan, ROC
}
\email{hsia@math.ntnu.edu.tw}

\author{T.~J.~Tucker}
\address{
Thomas Tucker\\
Department of Mathematics\\
University of Rochester\\
Rochester, NY 14627\\
USA
}
\email{ttucker@math.rochester.edu}

\begin{abstract}
Given a smooth projective curve $C$ defined over $\Qbar$ and given two elliptic surfaces $\mathcal{E}_1\lra C$ and $\mathcal{E}_2\lra C$ along with  sections $P_i,Q_i$ of $\mathcal{E}_i$ (for $i=1,2$), we prove that if there exist infinitely many $t\in C(\Qbar)$ such that for some integers $m_{1,t},m_{2,t}$, we have that $[m_{i,t}](P_{i})_t=(Q_{i})_t$ on $\mathcal{E}_i$ (for $i=1,2$), then at least one of the following conclusions must hold: either (i) there exists a nontrivial isogeny $\psi: \mathcal{E}_1\lra \mathcal{E}_2$ and also there exist nontrivial endomorphisms $\varphi_i$ of $\mathcal{E}_i$ (for $i=1,2$) such that $\varphi_2(P_2)=\psi(\varphi_1(P_1))$; or (ii) $Q_i$ is a multiple of $P_i$ for some $i=1,2$. A special case of our result answers a conjecture made by Silverman.
\end{abstract}

\thanks{ The research of the first author was partially supported by an NSERC Discovery grant. The second author was supported by MOST grant 104-2115-M-003-004-MY2. The third author was partially supported by NSF Grant DMS-0101636.}

\maketitle

\section{Introduction}
\label{sec:introduction}

In \cite{Ailon-Rudnick}, Ailon and Rudnick showed that for two
multiplicatively independent nonconstant polynomials
$a, b \in \bC[x]$, there is a nonzero polynomial $h\in \bC[x]$, depending on
$a$ and $b$ such that $\gcd(a^n - 1, b^n - 1) \mid h$ for all positive
integer $n$.  In this paper, we prove a similar result for elliptic
curves; instead of working with the multiplicative group
$\bG_m$, we work with the group law on an elliptic curve defined over
a function field.  The result of Ailon-Rudnick relies crucially on the
Serre-Ihara-Tate theorem (see \cite{lang}), while our result relies
crucially on recent Bogomolov type results for elliptic surfaces due to DeMarco and Mavraki \cite{DM}.

Throughout our article, we work with elliptic surfaces over $\Qbar$; more precisely, given a projective, smooth curve $C$ defined over $\Qbar$, an \emph{elliptic surface} $\cE/C$ is given by a morphism $\pi:\mathcal{E}\lra C$ over $\Qbar$ where the generic fiber of $\pi$ is an elliptic curve $E$ defined over $K = \Qbar(C)$, while for all but finitely many $t\in C(\Qbar)$, the fiber $\cE_t:=\pi^{-1}(\{t\})$ is an elliptic curve defined over $\Qbar$. Recall that a section $\sigma$ of $\pi$ (i.e. a map $\sigma:C\lra \mathcal{E}$ such that $\pi\circ \sigma = {\rm id}|_C$)  gives rise to a $K$-rational point of $E$. Conversly, a point $P\in E(K)$ corresponds to a section of $\pi$; if we need to indicate the dependence  with $P$,  we will  denote it by $\sigma_P$. So, for all but finitely many $t\in C(\Qbar)$, the intersection of the image of $\sigma_P$ in $\mathcal{E}$ with the fiber above $t$ is a point $P_t:=\sigma_P(t)$ on the elliptic curve $E_t$.
For any integer $k$, the multiplication-by-$k$ map $[k]$ on $E$ extends to a morphism on  $\mathcal{E}$; if there is no danger of confusion, we still denote the extension by $[k]$.

We prove the following result.
\begin{thm}
\label{main}
Let $\pi_i:\mathcal{E}_i\lra C$ be elliptic surfaces over a curve $C$ defined over $\Qbar$ with generic fibers $E_i$, and let  $\sigma_{P_i},\sigma_{Q_i}$ be sections of $\pi_i$ (for $i=1,2$) corresponding to points $P_i,Q_i\in E_i(\Qbar(C))$. If there exist infinitely many $t\in C(\Qbar)$ for which there exist some $m_{1,t},m_{2,t}\in\Z$ such that $[m_{i,t}]\sigma_{P_i}(t)= \sigma_{Q_{i}}(t)$ for $i=1,2$, then at least one of the following properties must hold:
\begin{enumerate}
\item[(i)] there exist isogenies $\varphi:E_1\lra E_2$ and  $\psi:E_2\lra E_2$  such that $\varphi(P_1) = \psi(P_2)$.
\item[(ii)] for some $i\in\{1,2\}$, there exists $k_i\in\Z$ such that $[k_i]P_i=Q_i$ on $E_i$.
\end{enumerate}
\end{thm}

We note here that, in contrast to similar results such as
\cite{Ailon-Rudnick} and \cite{HT}, the underlying multiplication
morphisms need not be isotrivial (that is, they need not be defined over
the field of constants in $k(C)$).

A special case of our result (when both $Q_1$ and $Q_2$ are the zero
elements) answers in the affirmative \cite[Conjecture~7]{Silverman-AR}; this is carried out in a more general setting  in our Proposition~\ref{prop:elliptic gcd} from Section~\ref{sec:elliptic gcd}.

Silverman's question \cite[Conjecture~7]{Silverman-AR} was motivated by work of Ailon-Rudnick
\cite{Ailon-Rudnick}, who showed that the greatest common divisor of
$a^n-1$ and of $b^n-1$ for multiplicatively independent polynomials
$a,b\in\C[T]$ has bounded degree (see also the generalization by
Corvaja-Zannier \cite{Corvaja-Zannier} along with the related results from \cite{CZ1, CZ2, CZ3}). In turn, the result of
Ailon-Rudnick was motivated by the work of Bugeaud-Corvaja-Zannier
\cite{BCZ} who obtained an upper bound for $\gcd(a^k-1,b^k-1)$ (as $k$
varies in $\N$) for given $a,b\in\Qbar$. On the other hand, Silverman
\cite{Silverman function field} showed that the degree of
$\gcd(a^m-1,b^n-1)$ could be quite large when $a,b\in
\overline{\mathbb{F}_p}[T]$; see also the authors' previous paper
\cite{GHT-GCD} where we study the $\gcd(a^m-1,b^n-1)$ when $a$ and $b$
are polynomials over arbitrary fields of positive characteristic,
along with other generalizations on the same theme. Finally, we
mention the work of Denis \cite{Denis-GCD} who studied the same
problem  of the greatest common divisor in the context of Drinfeld
modules. As hinted in \cite{Silverman-AR}, this \emph{greatest common divisor (GCD) problem} may
be studied in much higher generality; for example, if one knew the
result of DeMarco-Mavraki \cite{DM} (see Theorem~\ref{DM theorem}) in the
context of abelian varieties, then our method would extend to a
similar conclusion for arbitrary abelian schemes over a base curve.

Our Theorem~\ref{main} is related also to \cite[Theorem~1.1]{B-C} (see also the extension from \cite{B-C-2}) where it is  shown that given $n$ linearly independent sections $P_i$ on the Legendre elliptic family $y^2=x(x-1)(x-t)$, there are at most finitely many parameters $t$ such that  the points $(P_i)_t$ satisfy two independent linear relations on the corresponding elliptic curve. Therefore, a special case of the result by  Barroero and Capuano is that given sections $P_1,P_2,Q_1,Q_2$ on the Legendre elliptic surface, if these $4$ sections are linearly independent, then there are at most finitely many $t$ such that for some $m_t,n_t\in\mathbb{Z}$ we have that $[m_t](P_1)_t=(Q_1)_t$ and $[n_t](P_2)_t=(Q_2)_t$. However, in our Theorem~\ref{main} we obtain the same conclusion under the weaker hypothesis that $Q_i$ is not a multiple of  $P_i$ for $i=1,2$ and also that $P_1$ and $P_2$ are linearly independent.

We also note that a special case of our Theorem~\ref{main} bears a resemblance to the classical Mordell-Lang problem (see \cite{Faltings}). Indeed, with the notation as in Theorem~\ref{main}, assume there exist infinitely many $t\in C(\Qbar)$ such that for some $m_{t}\in\Z$ we have
\begin{equation}
\label{m and t}
[m_{t}](P_i)_t=(Q_i)_t\text{ for $i=1,2$.}
\end{equation}
Also assume there is no $m\in\Z$ such that $[m]P_i=Q_i$ for $i=1,2$. Then the conclusion of Theorem~\ref{main} yields the existence of isogenies $\varphi:E_1\lra E_2$ and $\psi:E_2\lra E_2$ such that $\varphi(P_1)=\psi(P_2)$. Thus, using that \eqref{m and t} holds for infinitely many $t\in C(\Qbar)$ we get that
\begin{equation}
\label{special Q}
\varphi(Q_1)=\psi(Q_2).
\end{equation}
Therefore, if we let $X\subset \mathcal{A}:=\mathcal{E}_1\times \mathcal{E}_2$ be the $1$-dimensional subscheme corresponding to the section $(Q_1,Q_2)$, and we let $\Gamma\subset \mathcal{A}$ be the cyclic subgroup spanned by $(P_1,P_2)$, then the existence of infinitely many $\gamma\in\Gamma$ such that for some $t\in C(\Qbar)$ we have $\gamma_t\in X$ yields that $X$ is contained in a proper algebraic subgroup of $\mathcal{A}$ (as given by the equation \eqref{special Q}). Such a statement can be viewed as a relative version of the classical Mordell-Lang problem; note that if $\cE_1$ and $\cE_2$ are constant elliptic surfaces with generic fibers $E_i^0$ defined over $\Qbar$,  while $\Gamma\subset (E_1^0\times E_2^0)(\Qbar)$, then this question is a special case of Faltings' theorem \cite{Faltings} (formerly known as the Mordell-Lang conjecture).  It is natural to ask whether the above relative version of the Mordell-Lang problem holds more generally when $\mathcal{A}\lra C$ is an arbitrary semiabelian scheme, $X\subset \mathcal{A}$ is a $1$-dimensional scheme and $\Gamma\subset \mathcal{A}$ is an arbitrary finitely generated group. This more general question is also related  to the bounded height problems studied in \cite{AMZ} in the context of pencils of finitely generated subgroups of $\mathbb{G}_m^n$.

In the next section of this paper, we review some preliminary
material.  Following that, in Section \ref{sec:elliptic}, we prove
Theorem \ref{main}.  The proof in the case of nonconstant sections is
quite similar to the proofs of the main results of \cite{Ailon-Rudnick} and
\cite{HT}, while the case of constant sections requires a different
argument.  In Section \ref{sec:elliptic gcd}, we give a positive
answer to Silverman's conjecture \cite[Conjecture~7]{Silverman-AR}.

\medskip

{\bf Acknowledgments.} We thank Myrto Mavraki and Joe Silverman for several useful conversations.

\section{Preliminaries}
\label{sec:notation}

From now on, we fix an elliptic surface $\pi:\mathcal{E}\lra C$, where $C$ is a projective, smooth curve defined over $\Qbar$. We denote by $E$ the generic fiber of $\mathcal{E}$; this is an elliptic curve defined over $\Qbar(C)$. For all but finitely many $t\in C(\Qbar)$, we have that $\mathcal{E}_t:=\pi^{-1}(\{t\})$ is an elliptic curve defined over $\Qbar$.

\subsection{Isotriviality}
We say that $\mathcal{E}$ is \emph{isotrivial} if the $j$-invariant of the generic fiber is a constant function (on $C$); for isotrivial elliptic surfaces  $\mathcal{E}$, all smooth fibers of $\pi$ are isomorphic (to the generic fiber $E$). If $\mathcal{E}$ is isotrivial, then there exists a finite cover $C'\lra C$ such that $\mathcal{E}':=\mathcal{E}\times_{C}C'$ is a \emph{constant (elliptic) surface} over $C'$, i.e., there exists an elliptic curve $E^0$ defined over $\Qbar$ such that $\mathcal{E}'=E^0\times_{{\rm Spec}(\Qbar)} C'$. Furthermore,  for a constant elliptic surface $E^0\times_{{\rm Spec}(\Qbar)} C'$, we say that $\sigma_P$ is a \emph{constant section} if $P\in E^0(\Qbar).$

\subsection{Canonical height on an elliptic surface}
For each $t\in C(\Qbar)$ such that $\cE_t$ is an elliptic curve, we let $\hhat_{\cE_t}$ be the N\'eron-Tate canonical height for the points in $\cE_t(\Qbar)$ (for more details, see \cite{Silverman-book-1}). There are two important properties of the canonical height which we will use:
\begin{itemize}
\item[(1)] $\hhat_{\cE_t}(P_t)=0$ if and only if $P_t$ is a torsion point of $\mathcal{E}_t$, i.e., there exists a positive integer $k$ such that $[k]P_t=0$.
\item[(2)] for each $k\in\Z$ we have that $\hhat_{\cE_t}([k]P_t)=k^2\cdot \hhat_{\cE_t}(P_t)$.
\end{itemize}

Also, we let $\hhat_{E}$ be the N\'eron-Tate canonical height on the generic fiber $E$ constructed with respect to the Weil height on the function field $\Qbar(C)$ (for more details, see \cite{Silverman-book-2}). Property~(2) above holds also on the generic fiber, i.e., $\hhat_{E}([k]P)=k^2\cdot \hhat_{E}(P)$. On the other hand, property~(1) above holds only if $\mathcal{E}$ is non-isotrivial. Furthermore, if $\mathcal{E}=E\times_C C$ is a constant family (where $E$ is an elliptic curve defined over $\Qbar$), then for any $P\in E(\Qbar(C))$, we have that $\hhat_{E}(P)=0$ if and only if $P\in E(\Qbar)$.

\subsection{Variation of the canonical height}
We let $h_C$ be a given Weil height for the points in $C(\Qbar)$ corresponding to a divisor of degree $1$ on $C$.

Let $P$ be a section of the elliptic surface $\mathcal{E}\lra C$. As before, we identify this section with its intersection (which we also call $P$) with the generic fiber of $\mathcal{E}$; also, for all but finitely many $t\in C(\Qbar)$, we have that the intersection of the image of $\sigma_P$ in $\mathcal{E}$  with the fiber above $t$ is a point $P_t$ on the elliptic curve $\cE_t$.
The following important fact will be used in our proof (see  \cite{CS}):
\begin{equation}
\label{equation variation}
\lim_{h_C(t)\to\infty}\frac{\hhat_{\cE_t}(P_t)}{h_C(t)} = \hhat_{E}(P).
\end{equation}
Furthermore, the following more precise result holds, as proven by Silverman \cite{Silverman-3},
\begin{equation}
\label{equation variation precise}
\hhat_{\cE_t}(P_t)=h_{C,\eta(P)}(t) + O_P(1),
\end{equation}
where $\eta(P)$ is a divisor on $C$ of degree equal to $\hhat_{E}(P)$ and $h_{C,\eta(P)}$ is a given Weil height for the points in $C(\Qbar)$  corresponding to the divisor $\eta(P)$, while the implicit constant from the term $O_P(1)$ is only dependent on the section $P$, but not on $t\in C(\Qbar)$.

\subsection{Points of small height on sections}

We will use the following result of DeMarco-Mavraki \cite[Theorem~1.4]{DM} who extends \cite{DWY} (and in turn, uses the extensive analysis from \cite{Silverman-3} regarding the variation of the canonical height in an elliptic fibration). We also note that the case of isotrivial elliptic curves from Theorem~\ref{DM theorem} was previously proven by Zhang \cite{Zhang}, as part of Zhang's famous proof of the classical Bogomolov conjecture.

\begin{thm}[DeMarco-Mavraki \cite{DM}]
\label{DM theorem}
Let $\mathcal{E}_1,\mathcal{E}_2$ be elliptic fibrations over the same $\Qbar$-curve $C$. Let $P_i$ be a section of $\mathcal{E}_i$ (for $i=1,2$) with the property that there exists an infinite sequence $\{t_n\}\subset C(\Qbar)$ such that
$$\lim_{n\to\infty} \hhat_{\left(\mathcal{E}_i\right)_{t_n}}\left((P_i)_{t_n}\right) = 0\text{ for }i=1,2.$$
Then there exist group homomorphisms $\phi:\mathcal{E}_1\lra \mathcal{E}_2$ and $\psi:\mathcal{E}_2\lra \mathcal{E}_2$, not both trivial, such that $\phi(P_1)=\psi(P_2)$.
\end{thm}

\section{Proof of our main result}
\label{sec:elliptic}

Propositions~\ref{prop:small points} and \ref{prop:small points 2} are key to our proof.
\begin{prop}
\label{prop:small points}
Let $C$ be a projective, smooth curve defined over $\Qbar$, and let $h_C(\cdot)$ be a Weil height for the algebraic points of $C$ corresponding to a divisor of degree $1$.
Let $P$ and $Q$ be sections of an elliptic surface $\pi:\mathcal{E}\lra C$ with generic fiber $E$, and assume there exists no $k\in\Z$ such that $[k]P=Q$. In addition, assume $\hhat_{E}(P)>0$. If there exists an infinite sequence $\{t_i\}\subset C(\Qbar)$ such that for each $i\in\N$ there exists some $m_i\in\Z$ such that $[m_i]P_{t_i}=Q_{t_i}$, then $h_C(t_i)$ is uniformly bounded and  $\lim_{i\to\infty}\hhat_{\cE_{t_i}}(P_{t_i})=0$.
\end{prop}

\begin{proof}
Since $[m_i]P_{t_i}=Q_{t_i}$, we get that
\begin{equation}
\label{height i}
m_i^2\cdot \hhat_{\cE_{t_i}}(P_{t_i})=\hhat_{\cE_{t_i}}(Q_{t_i}).
\end{equation}

Since $[k]P\ne Q$ for any $k\in\Z$ and the sequence $\{t_i\}$ is infinite, then
\begin{equation}
\label{infinite m i}
\lim_{i\to\infty}|m_i|=\infty.
\end{equation}

We claim first that $h_C(t_i)$ is uniformly bounded. Indeed, assuming (at the expense, perhaps, of replacing $\{t_i\}$ by an infinite subsequence) that $\lim_{i\to\infty}h_C(t_i)=\infty$, equation \eqref{equation variation} coupled with equations \eqref{height i} and \eqref{infinite m i} yields a contradiction. To see this, we divide both sides of \eqref{height i} by $h_C(t_i)$ and then take limits.  Because  $\hhat_{E}(P)>0$, equation \eqref{infinite m i} yields that the left hand side equals
\begin{equation}
\label{left hand limit infinite}
\lim_{i\to\infty} m_i^2\cdot \frac{\hhat_{\cE_{t_i}}(P_{t_i})}{h_C(t_i)}=\infty,
\end{equation}
while the right hand side equals
\begin{equation}
\label{right hand side finite}
\lim_{i\to\infty} \frac{\hhat_{\cE_{t_i}}(Q_{t_i})}{h_C(t_i)} = \hhat_{E}(Q)<\infty,
\end{equation}
which is a contradiction. So, indeed, we must have that $h_C(t_i)$ is uniformly bounded.

Next we prove that also $\hhat_{\cE_{t_i}}(Q_{t_i})$ is uniformly bounded. Using \eqref{equation variation precise} (see \cite{Silverman-3}), we know that there exists a divisor $\eta(Q)$ of $C$ of degree equal to $\hhat_{E}(Q)$ such that
\begin{equation}
\label{variation 3}
\hhat_{\cE_t}(Q_t)=h_{C,\eta(Q)}(t) + O(1),
\end{equation}
where $h_{C,\eta(Q)}$ is a Weil height on $C(\Qbar)$ corresponding to the divisor $\eta(Q)$. Since $h_C$ is a Weil height associated to a divisor $D$ on $C$ of degree $1$, then for any positive integer $m>\deg(\eta(Q))$, the divisor $D_1:=mD- \eta(Q)$ has positive degree and therefore, it is ample. Then  \cite[Proposition~B.3.2]{Hindry-Silverman} yields that any Weil height $h_{C,D_1}$ associated to the divisor $D_1$ satisfies $h_{C,D_1}(t)\ge O(1)$ for all $t\in C(\Qbar)$. So,
\begin{equation}
\label{height machine equation}
mh_C(t) + O(1)\ge h_{C,\eta(Q)}(t)\text{ for }t\in C(\Qbar).
\end{equation}
Therefore $h_{C,\eta(Q)}(t_i)$ is uniformly bounded (since $h_C(t_i)$ is uniformly bounded). Then equation \eqref{variation 3} provides the desired claim that
\begin{equation}
\label{height of R t bounded}
\hhat_{\cE_{t_i}}(Q_{t_i})\text{ is bounded as }i\to\infty.
\end{equation}
Combining equations \eqref{height i}, \eqref{infinite m i} and \eqref{height of R t bounded}, we finish the proof of Proposition~\ref{prop:small points}.
\end{proof}

\begin{prop}
\label{prop:small points 2}
Let $P$ and $Q$ be sections of a constant elliptic fibration $\pi:\mathcal{E}\lra C$, and assume there exists no $k\in\Z$ such that $[k]P=Q$. In addition, assume $P$ is a non-torsion, constant section. If there exists an infinite sequence $\{t_i\}\subset C(\Qbar)$ such that for each $i\in\N$ there exists some $m_i\in\Z$ such that $[m_i]P_{t_i}=Q_{t_i}$, then $\lim_{i\to\infty} h_C(t_i)=\infty$.
\end{prop}

\begin{proof}
We have that each fiber $\cE_{t_i}$ is isomorphic to the generic fiber $E^0$, and so, because $P$ is a constant section,
\begin{equation}
\label{same canonical height on fibers}
\hhat_{\cE_{t_i}}(P_{t_i})=\hhat_{E^0}(P^0),
\end{equation}
where $P^0$ is the intersection of $P$ with the generic fiber and $\hhat_{E^0}(\cdot )$ is the N\'eron-Tate canonical height of the elliptic curve $E^0$ defined over $\Qbar$ (i.e., it is not the canonical height on the generic fiber of $\mathcal{E}$ seen as an elliptic curve defined over the function field $\Qbar(C)$).

Furthermore, since $P^0$ is not a torsion point of $E^0$, then $\hhat_{E^0}(P^0)>0$. Thus, from the equality $[m_i]P_{t_i}=Q_{t_i}$, along with equation \eqref{same canonical height on fibers} coupled with the fact that $|m_i|\to\infty$ (because $[k]P\ne Q$ for all integers $k$), we get that
\begin{equation}
\label{large height t i 0}
\hhat_{\cE_{t_i}}(Q_{t_i})=m_i^2\hhat_{E^0}(P^0)\to\infty.
\end{equation}
Then, using \eqref{equation variation precise}, we have
\begin{equation}
\label{variation height R}
\hhat_{\cE_{t_i}}(Q_{t_i})=h_{C,\eta(Q)}(t_i) + O(1),
\end{equation}
where $h_{C,\eta(Q)}$ is a Weil height on $C$ corresponding to a divisor $\eta(Q)$. So, equations \eqref{large height t i 0} and \eqref{variation height R} yield $h_{C,\eta(Q)}(t_i)\to\infty$ and thus, $h_C(t_i)\to\infty$ (see \cite[Proposition~B.3.5]{Hindry-Silverman}, along with our similar argument from the proof of Proposition~\ref{prop:small points}).
\end{proof}

Now we can prove our main result.
\begin{proof}[Proof of Theorem~\ref{main}.]
First we note that if $P_i$ is a torsion section (for some $i\in\{1,2\}$), then conclusion~(ii) holds trivially since then we would get that there exist infinitely many $t\in C(\Qbar)$ such that $(Q_i)_t=[k](P_i)_t$ for the same integer $k$. So, from now on, we assume that both $P_1$ and $P_2$ are nontorsion sections on $\mathcal{E}_1$, respectively $\mathcal{E}_2$.  In particular, this means that if $\hhat_{E_i}(P_i)=0$, then $\mathcal{E}_i$ must be an isotrivial elliptic surface.

We assume there exists an infinite sequence $\{t_i\}\subset C(\Qbar)$ such that for each $i\in\N$ there exist $m_{i,1},m_{i,2}\in\Z$ with the property that $[m_{i,1}](P_1)_{t_i}=(Q_1)_{t_i}$ and also $[m_{i,2}](P_2)_{t_i}=(Q_2)_{t_i}$. In addition, we assume conclusion~(ii) does not hold, i.e., there is no $m\in\Z$ such that $[m]P_i=Q_i$ for some $i\in\{1,2\}$.
We split our analysis into two cases.

{\bf Case 1.} $\hhat_{E_i}(P_i)>0$ for each $i=1,2$.

Applying then Proposition~\ref{prop:small points} to the sections $P_i$ and $Q_i$, we get that
\begin{equation}
\label{twice small points}
\lim_{i\to\infty} \hhat_{(\cE_1)_{t_i}}\left((P_1)_{t_i}\right)=\lim_{i\to\infty} \hhat_{(\cE_2)_{t_i}}\left((P_2)_{t_i}\right) =0.
\end{equation}
Equation \eqref{twice small points} along with Theorem~\ref{DM theorem} yields that conclusion~(i) must hold in Theorem~\ref{main}. Note that we obtain in this case that the morphisms $\varphi:E_1\lra E_2$ and $\psi:E_2\lra E_2$ from the conclusion of Theorem~\ref{DM theorem} are \emph{both} isogenies since $P_1$ and $P_2$ are nontorsion sections.

{\bf Case 2.} Either $\hhat_{E_1}(P_1)=0$ or $\hhat_{E_2}(P_2)=0$.

Without loss of generality, we assume $\hhat_{E_1}(P_1)=0$. Therefore (since $P_1$ is not torsion) $\mathcal{E}_1$ is an isotrivial elliptic surface, and furthermore, at the expense of replacing $C$ by a finite cover (and also base extending $\mathcal{E}_1$ and respectively $\mathcal{E}_2$), we may assume that $\mathcal{E}_1$ is a constant family. Thus, $\mathcal{E}_1=E_1^0\times_C C$ for some elliptic curve $E_1^0$ defined over $\Qbar$. Then also $P$ is a constant (nontorsion) section, because $\hhat_{\mathcal{E}_1}(P_1)=0$. Finally, we let $h_C(\cdot )$ be a Weil height for the algebraic points of $C$ with respect to a divisor of degree $1$.

If $\hhat_{E_2}(P_2)>0$, then Proposition~\ref{prop:small points} applied to $P_2$ and $Q_2$ yields that $h_C(t_i)$ is uniformly bounded, which contradicts the conclusion of Proposition~\ref{prop:small points 2} applied to $P_1$ and $Q_1$. Therefore, we must have that $\hhat_{E_2}(P_2)=0$ and therefore, also $\mathcal{E}_2$ is an isotrivial elliptic surface. At the expense of (yet another) base extension, we may assume that also $\mathcal{E}_2=E_2^0\times C$ is a constant fibration. Then $P_2$ is a constant, nontorsion section on $\mathcal{E}_2$. We let $P_i^0$ be the intersection of $P_i$ (for $i=1,2$) with the generic fiber of $\mathcal{E}_i$.

Now, if either $Q_1$ or $Q_2$ is also a constant section, then we get a contradiction since we assumed conclusion~(ii) does not hold. Indeed, if for some $i=1,2$ we have that both $P_i$ and $Q_i$ are constant sections on the constant elliptic surface $\mathcal{E}_i$, then the existence of a point $t\in C(\Qbar)$ such that for some $k\in \mathbb{Z}$ we have $[k](P_i)_t=(Q_i)_t$ yields that actually $[k]P_i=Q_i$ on $\cE_i$. So, we may assume that $Q_1$ and $Q_2$ are both nonconstant sections on $\mathcal{E}_1$, respectively $\mathcal{E}_2$. Then there is a (neither vertical, nor horizontal) curve $X\subset E_1^0\times E_2^0$ containing all points $\left((Q_1)_t, (Q_2)_t\right)$ for $t\in C(\Qbar)$. Furthermore, our hypothesis yields that this curve $X$ intersects the subgroup $\Gamma\subset E_1^0\times E_2^0$ spanned by the points $(P^0_1,0)$ and $(0,P_2^0)$ in an infinite set. Then the classical Mordell-Lang conjecture (proven by Faltings \cite{Faltings}) yields that $X$ itself is a coset of an algebraic subgroup of $E_1^0\times E_2^0$. Hence, because $X$ projects dominantly onto each coordinate, there exists a nontrivial isogeny $\tau:E_1^0\lra E_2^0$, and also there exist endomorphisms $\phi_i$ of $E_i^0$, not both trivial, such that $\tau(\phi_1(Q_1))=\phi_2(Q_2)$. Then, using (for any $i$ such that $m_{i,1}$ and $m_{i,2}$ are nonzero) that
$$[m_{i,1}]P_1^0=(Q_1)_{t_i}\text{ and }[m_{i,2}]P_2^0=(Q_2)_{t_i}$$
along with the fact that $\tau\left(\phi_1\left((Q_1)_{t_i}\right)\right)=\phi_2\left((Q_2)_{t_i}\right)$, we obtain the conclusion in Theorem~\ref{main} with $\varphi:=\tau\circ [m_{i,1}]\circ \phi_1$ and $\psi:=[m_{i,2}]\circ \phi_2$. Finally, note that since $P_1$ and $P_2$ are non-torsion, then also $\varphi$ and $\psi$ are dominant morphisms.
\end{proof}

\section{Common divisors of elliptic sequences}
\label{sec:elliptic gcd}

In this section, we apply Theorem~\ref{main} to prove Silverman's conjecture~\cite[Conjecture~7]{Silverman-AR} concerning common divisors of elliptic sequences; actually, our Proposition~\ref{prop:elliptic gcd} provides a slightly more general answer than \cite[Conjecture~7]{Silverman-AR}. We first recall the terminology and notation  from~\cite{Silverman-AR} that we will use in this section.

Let $k$ be an algebraically closed field of characteristic $0$.
Let $C$ be a  smooth projective curve defined over $k$ and let $K = k(C)$ be the function field of $C$. For any point  $\gamma\in C(k)$, we let $\ord_\gamma(D)$ denote the coefficient of $\gamma$ in $D \in \Div(C).$ The {\em greatest common divisor} for any two effective divisors $D_1, D_2 \in \Div(C)$ is defined as
$$
\GCD(D_1, D_2) = \sum_{\gamma\in C} \min\{\ord_\gamma(D_1),
\ord_\gamma(D_2)\}\cdot (\gamma) \in \Div(C) .
$$

For an elliptic curve $E$ defined over $K$, let $\pi : \cE\lra C$ be an elliptic surface whose generic fiber is $E$ and let $P\in E(K)$. Recall that the section corresponding to $P$ is denoted by $\sigma_P : C \to \cE$. We denote the image of $\sigma_P$ by $\bar{P}:= \sigma_P(C) \subset \cE.$

Let $E_1$ and $E_2$ be elliptic curves defined over $K$, let $\cE_i/C$
be elliptic surfaces with generic fibers  $E_i$, and let
$P_i\in E_i(K)$ for $i=1, 2.$ The greatest common divisor of $P_1$ and $P_2$ is given by
$$
\GCD(P_1, P_2) = \GCD\left(\sigma_{P_1}^{\ast}(\Obar_{\cE_1}),
  \sigma_{P_2}^{\ast}(\Obar_{\cE_2})\right).
$$
where $\Obar_{\cE_i} := \sigma_{O_i}(C)$ is the zero section on $\cE_i$ corresponding to the identity $O_i$ of $E_i, i =1, 2.$
Thus, for any given  $Q_i\in E_i(K)$, $\GCD(P_1 - Q_1, P_2 - Q_2)$ is the greatest common divisor of the two points $P_i-Q_i \in E_i$ for $i=1, 2$.
In the following, points $P_1$ and $P_2$ are called \emph{($K$-)dependent} if there are morphisms $\varphi : E_1\lra E_2$ and $\psi : E_2 \lra E_2$ not both trivial  such  that $\varphi(P_1) = \psi(P_2)$; otherwise they are called \emph{independent}. Note that if one of $P_1$ and $P_2$ is a torsion point, then they are automatically dependent.

Motivated by  Ailon-Rudnick's result~\cite{Ailon-Rudnick}, Silverman conjectured that an elliptic analogue
also exists. For the convenience of the reader, we recall his conjecture as follows.

\begin{conjecture}[Silverman~\protect{\cite[Conjecture~7]{Silverman-3}}]
  \label{conj:char 0 common divisor}
Let $K = k(C)$ be the function field of a smooth projective curve $C$ over an algebraically closed field $k$ of characteristic $0$,  let   $E_1/K$ and $E_2/K$ be elliptic curves, and let $P_1 \in E_1(K)$ and $P_2 \in
E_2(K)$ be $K$-independent points.

\begin{enumerate}
\item There is a constant $c = c(K, E_1, E_2, P_1, P_2)$ so that
  $$
  \deg \GCD([n_1] P_1, [n_2] P_2) \le c \quad \text{for all $n_1, n_2 \ge
1$.}
$$

\item Further, there is an equality
  $$
  \GCD([n] P_1, [n] P_2) = \GCD (P_1, P_2) \quad \text{for infinitely
    many $n \ge 1.$}
  $$
\end{enumerate}
\end{conjecture}

\begin{remark}
Silverman~\cite[Theorem~8]{Silverman-3} showed that Conjecture~\ref{conj:char 0 common divisor} is true provided that both $E_1$ and $E_2$ have constant $j$-invariant.
\end{remark}

As an application of Theorem~\ref{main}, we prove that  Conjecture~\ref{conj:char 0 common divisor} holds (even in a slightly stronger form); we strengthen further the conclusion from Conjecture~\ref{conj:char 0 common divisor} when $k=\Qbar$.


\begin{prop}
\label{prop:elliptic gcd}
Let $k$ be an algebraically closed field of characteristic $0$.
Let $C$ be a smooth projective curve defined over $k$, let $K = k(C)$ and let $E_i/K, i=1, 2,$ be elliptic curves defined over $K.$ Let $P_i, Q_i\in E_i(K)$ for $i=1, 2$ and furthermore, assume that $P_1$ and $P_2$ are $K$-independent.
\begin{enumerate}
\item If $k=\Qbar$, then there exists an effective divisor $D\in \Div(C)$ such that
  \[
\GCD\left([n_1] P_1 - Q_1, [n_2] P_2 - Q_2\right) \le D
\]
for all integers $n_i$ such that $[n_i] P_i \ne Q_i, i=1, 2.$

\item For an arbitrary algebraically closed field $k$ of characteristic $0$,  there exists an effective divisor $D_0\in \Div(C)$ such that
  \[
\GCD\left([n_1] P_1 , [n_2] P_2 \right) \le D_0
\]
for all nonzero integers $n_i$.

\item The set
  \[
\left\{n\ge 1 : \GCD\left([n] P_1 , [n] P_2\right) = \GCD\left(P_1, P_2\right)\right\}
\]
has positive density in $\mathbb{N}$.

\item For all but finitely many primes $q$, we have that $\GCD\left([q] P_1 , [q] P_2\right) = \GCD\left(P_1, P_2\right)$.
\end{enumerate}

\end{prop}

\begin{remark}
The conclusion of Proposition~\ref{prop:elliptic gcd}~(i) for an arbitrary algebraically closed field $k$ of characteristic $0$ would follow from our method once the validity of DeMarco-Mavraki's result \cite{DM} (see Theorem~\ref{DM theorem}) is extended over function fields. In turn, their result is contingent on establishing the smooth variation of the canonical height in fibers of an elliptic surface defined over a function field (over $\Qbar$).
\end{remark}

The proof of Proposition~\ref{prop:elliptic gcd} relies on Theorem~\ref{main} and the following lemma which is a variant of~\cite[Lemma~4]{Silverman-AR}  bounding $\ord_\gamma(\sigma_{[n]P}^\ast(\Obar_{\cE}))$ for $\gamma\in C$ and all integers $n\ne 0.$

\begin{lemma}
\label{lem:multiplicity}
Let $k$ be an algebraically closed field of characteristic $0$.
Let $E$ be an elliptic curve defined over $k$ and let $\cE\lra C$ be an  elliptic surface whose generic fiber is $E$. Let $\gamma \in C(k)$ and let $P, Q\in E(k(C))$ be given.
There exists a constant $m = m(\gamma, E, P, Q)$ such that $\ord_{\gamma}(\sigma_{[n]P}^{\ast}(\bar{Q}))\le m$ for all  integers $n$ such that $[n] P \ne Q.$
\end{lemma}

\begin{proof}
    Observe that $\ord_\gamma(\sigma_{[n] P}^\ast(\bar{Q})) \ge 1$ if and only if $\sigma_{[n]P}(\gamma) =  \sigma_{Q}(\gamma).$ Moreover,  $\sigma_Q(\gamma)$ is a torsion point of $\cE_\gamma$ if and only if there are more than one $n$ such that $\ord_\gamma(\sigma_{[n] P}^\ast(\bar{Q})) \ge 1.$

    It suffices to prove the assertion when $\ord_\gamma(\sigma_{n P}^\ast(\bar{Q})) \ge 1$ for more than one integer $n.$
    Thus, we assume that $ \sigma_{Q}(\gamma)$ is a torsion point of $\cE_\gamma.$  Let $\ell$ be the order of $\sigma_{Q}(\gamma)$ and that $\ord_\gamma(\sigma_{n P}^\ast(\bar{Q})) \ge 1$ for some integer $n$ such that $[n]P \ne Q$.  It follows that $\ord_\gamma(\sigma_{n P}^\ast(\bar{Q}))$ is finite and that
\begin{equation}
\label{the zero element}
\sigma_{[\ell n]P}(\gamma) = [\ell] \sigma_{[n]P}(\gamma)= [\ell] \sigma_{Q}(\gamma) = O_{\cE_\gamma},
\end{equation}
which is the zero element for the elliptic curve $\cE_\gamma$.

   If $Q$ is the zero element of $E$,  then it follows from~\cite[Lemma~4]{Silverman-AR} that the value of  $\ord_\gamma(\sigma_{[n] P}^\ast(\bar{O}_\cE))$ is bounded independent of $n\ne 0$ and we are done in this case.

   Assume that $Q \ne O$. Then \eqref{the zero element} yields the inequality  $$\ord_\gamma(\sigma_{[n] P}^\ast(\bar{Q})) \le \ord_\gamma(\sigma_{[\ell n] P}^\ast(\bar{O}_\cE)).$$ Hence,  we know that  $\ord_\gamma(\sigma_{[n] P}^\ast(\bar{Q})) $ is bounded independent of $n\ne 0$ (and $n$ such that $[n] P \ne Q$). As $Q\ne O$,  we also have that  $\ord_\gamma(\sigma_{n P}^\ast(\bar{Q}))$ is finite if $n=0$.
Thus we obtain that  $\ord_\gamma(\sigma_{n P}^\ast(\bar{Q}))$ is bounded independent of $n$ such that $[n]P \ne Q$, which concludes our proof.
\end{proof}

\begin{proof}[Proof of Proposition~\ref{prop:elliptic gcd}]
We first prove part~(i) in Proposition~\ref{prop:elliptic gcd}. So, for each $\gamma\in C(\Qbar)$, let $m_{i, \gamma}$ be an upper bound for
  $\ord_{\gamma}(\sigma_{[n]P_i}^{\ast}(\bar{Q_i}))$ as in Lemma~\ref{lem:multiplicity}.
  Set $m_{\gamma} = \min\{m_{1,\gamma}, m_{2, \gamma}\}$.
Since $P_1$ and $P_2$ are independent, by Theorem~\ref{main} we may take $m_\gamma = 0$ for all but finitely many points $\gamma\in C(\Qbar).$ Put
\[
 D := \sum_{\gamma\in C(\Qbar)}\, m_\gamma (\gamma).
\]
Then, $D$ is an effective divisor of $C$. Now it follows directly from Lemma~\ref{lem:multiplicity} that $\GCD([n_1] P_1 - Q_1, [n_2] P_2 - Q_2) \le D $ for all $n_i$ such that $[n_i]P\ne Q_i$ for both $i=1, 2.$

For the proof of part~(ii) in Proposition~\ref{prop:elliptic gcd},  we let $Q_i=O_i$ be the zero element of $E_i$ for $i=1, 2.$ If $k=\Qbar$, then the result follows immediately from part~(i). Now, for the general case, we note that it suffices to prove the existence of at most finitely many $t\in C(k)$ such that both $(P_1)_t$ and $(P_2)_t$ are torsion points on the elliptic fiber $\cE_t$; indeed, the fact that the multiplicity of each such $t$ appearing in a divisor $\GCD([n_1]P_1, [n_2]P_2)$ is bounded follows exactly as in the proof of part~(i), using Lemma~\ref{lem:multiplicity}. On the other hand, if there exist infinitely many $t\in C(k)$ such that both $(P_1)_t$ and $(P_2)_t$ are torsion yields (according to \cite[Theorem,~p.~117]{M-Z}) that $P_1$ and $P_2$ are related, contradiction.

The conclusion in part~(iii) in Proposition~\ref{prop:elliptic gcd} follows almost verbatim from the proof of~\cite[Theorem~8~(b)]{Silverman-AR}. Essentially, the argument is as follows. For each of the finitely many $\gamma\in C(k)$ which does not appear in the support of $\GCD(P_1,P_2)$, but for which there exists some positive integer $n$ such that $\gamma$ is contained in the support of the divisor $\GCD([n]P_1, [n]P_2)$, or equivalently,
\begin{equation}
\label{n it holds}
\text{the divisor }\GCD([n]P_1, [n]P_2)-(\gamma)\text{ is effective,}
\end{equation}
we let $n_\gamma$ be the smallest such positive integer $n$ for which \eqref{n it holds} holds. Then it is easy to see that $\gamma$ is in the support of $\GCD([n]P_1,[n]P_2)$ if and only if $n_\gamma\mid n$. Also, for each of these points $\gamma$ which are not in the support of $\GCD(P_1,P_2)$, we have that $n_\gamma>1$. Then for any positive integer $n$ which is not divisible by any of the finitely  many integers $n_\gamma$, we have that
$$\GCD([n]P_1,[n]P_2)=\GCD(P_1,P_2).$$

The conclusion in part~(iv) in Proposition~\ref{prop:elliptic gcd} follows from the proof of part~(iii) since for all primes $q$ which do not divide any of the finitely many numbers $n_\gamma>1$, we have that $\GCD([q]P_1,[q]P_2)=\GCD(P_1,P_2)$.  
\end{proof}


\bibliographystyle{amsalpha}

\begin{thebibliography}{XXX999}

\bibitem[AR04]{Ailon-Rudnick}
N.~Ailon and Z.~Rudnick, \emph{Torsion points on curves and common divisors of $a^k-1$ and $b^k-1$}, Acta Arith. \textbf{113} (2004), no.~1, 31-38.

\bibitem[AMZ]{AMZ}
F.~Amoroso, D.~Masser, and U.~Zannier, \emph{Bounded height in pencils of finitely generated groups}, preprint, arxiv 1509.04963.

\bibitem[BC16]{B-C}
F.~Barroero and L.~Capuano, \emph{Linear relations in families of powers of elliptic curves}, Algebra \& Number Theory {\bf 10} (2016), 195-214.

\bibitem[BC]{B-C-2}
F.~Barroero and L.~Capuano, \emph{Unlikely intersections in families of products of elliptic curves and the multiplicative group}, arxiv 1606.02063.

\bibitem[BG06]{BG}
E.~Bombieri and W.~Gubler, \emph{Heights in Diophantine geometry}, New
  Mathematical Monographs, vol.~4, Cambridge University Press, Cambridge, 2006.

\bibitem[BCZ03]{BCZ}
Y. Bugeaud, P. Corvaja and U. Zannier, \emph{An upper bound for the G.C.D. of $a^n-1$ and $b^n-1$}, Math. Z. \textbf{243} (2003), 79-84.

\bibitem[CS93]{CS}
G.~S. Call and J.~Silverman, \emph{Canonical heights on varieties with
  morphism}, Compositio Math. \textbf{89} (1993), 163--205.

\bibitem[CZ08]{CZ1}
P.~Corvaja and  U.~Zannier, \emph{Some cases of Vojta's conjecture on integral points over function fields}, J.  Algebraic  Geom. {\bf 17}
(2008), 295-333; Addendum, Asian J. Math. {\bf 14} (2010), 581-584.

\bibitem[CZ11]{CZ2}
P.~Corvaja and U.~Zannier,  \emph{An $abcd$ theorem over function fields and applications},
Bull. Soc. Math. France {\bf 139} (2011), 437-454.

\bibitem[CZ13a]{CZ3}
P.~Corvaja and U.~Zannier, \emph{Algebraic hyperbolicity of ramified covers of $\mathbb{G}_m^2$ (and integral points on affine subsets of $\mathbb{P}^2$)},
J. Differential Geom. {\bf 93} (2013), 355-377.

\bibitem[CZ13b]{Corvaja-Zannier}
P.~Corvaja and U.~Zannier, \emph{Greatest common divisors of
$u-1$, $v-1$
in positive characteristic and rational points on curves over finite fields},
J. Eur. Math. Soc. (JEMS) \textbf{15} (2013), no.~5, 1927-1942.


\bibitem[DM]{DM}
L.~DeMarco and N.~M.~Mavraki,
\emph{Variation of canonical height and equidistribution}, preprint, arxiv 1701.07947v1.

\bibitem[DWY16]{DWY}
L.~DeMarco, X.~Wang, and H.~Ye,
\emph{Torsion points and the Latt\'{e}s family}, Amer. J. Math. \textbf{138} (2016), no.~3, 697-732.

\bibitem[Den11]{Denis-GCD}
L.~Denis, \emph{Facteurs communs et torsion en caract\'{e}ristique non nulle},  J. Th\'{e}or. Nombres Bordeaux \textbf{23} (2011), no.~2, 347-352.

\bibitem[Fal94]{Faltings}
G.~Faltings, \emph{The general case of S. Lang's conjecture}, Barsotti Symposium in Algebraic Geometry (Abano Terme, 1991), 175-182, Perspect. Math. \textbf{15}, Academic Press, San Diego, CA, 1994.

\bibitem[GHT17]{GHT-GCD} D.~Ghioca, L.-C.~Hsia, and T.~J.~Tucker,
  \emph{On a variant of the Ailon--Rudnick theorem in finite
    characteristic}, New York J. Math. {\bf 23} (2017), 213-225.


\bibitem[HT]{HT} L.-C.~Hsia and T.~J.~Tucker, \emph{Greatest common
    divisors of iterates of polynomials}, preprint, arxiv 1611.04115.


\bibitem[HS00]{Hindry-Silverman}
M.~Hindry and J.~H.~Silverman, \emph{Diophantine geometry. An introduction},  Graduate Texts in Mathematics \textbf{201}, Springer-Verlag, New York, 2000,  xiv+558 pp.

\bibitem[Lan65]{lang}
S.~Lang, \emph{Division points on curves}, Ann. Mat. Pura Appl. (4) \textbf{70}
  (1965), 229--234.

\bibitem[MZ14]{M-Z}
D.~Masser and U.~Zannier, \emph{Torsion points on families of products of elliptic curves}, Adv. Math. {\bf 259} (2014), 116-133.



\bibitem[Sil83]{Silverman-1}
J.~H.~Silverman, \emph{Heights and specialization map for families of abelian varieties},  J. Reine Angew. Math. \textbf{342} (1983), 197-211.


\bibitem[Sil86]{Silverman-book-1}
J.~H.~Silverman, \emph{The arithmetic of elliptic curves}, Second edition, Graduate Texts in Mathematics \textbf{106}, Springer-Verlag, New York, 1986, xii+400 pp.

\bibitem[Sil94a]{Silverman-book-2}
J.~H.~Silverman, \emph{Advanced topics in the arithmetic of elliptic curves},
Graduate Texts in Mathematics \textbf{151}, Springer-Verlag, New York, 1994,  xiv+525 pp.

\bibitem[Sil94b]{Silverman-3}
J.~H.~Silverman, \emph{Variation of the canonical height on elliptic surfaces. III. Global boundedness properties},  J. Number Theory \textbf{48} (1994), no.~3, 330-352.

\bibitem[Sil04a]{Silverman function field}
J.~H.~Silverman, \emph{Common divisors of $a^n-1$ and $b^n-1$ over function fields}, New York J. Math. \textbf{10} (2004), 37-43.

\bibitem[Sil04b]{Silverman-AR}
J.~H.~Silverman, \emph{Common divisors of elliptic divisibility sequences over function fields}, Manuscripta Math. \textbf{114} (2004), no.~4, 431-446.

\bibitem[Zha98]{Zhang}
S.~Zhang, \emph{Equidistribution of small points on abelian varieties}, Ann. of Math. (2) \textbf{147} (1998), no.~1, 159-165.

\end{thebibliography}

\def\cprime{$'$} \def\cprime{$'$} \def\cprime{$'$} \def\cprime{$'$}
\providecommand{\bysame}{\leavevmode\hbox to3em{\hrulefill}\thinspace}
\providecommand{\MR}{\relax\ifhmode\unskip\space\fi MR }
\providecommand{\MRhref}[2]{%
  \href{http://www.ams.org/mathscinet-getitem?mr=#1}{#2}
}
\providecommand{\href}[2]{#2}

\end{document}